\documentclass[12pt]{article}
\usepackage{latexsym,amssymb,amsmath,amsfonts,amsthm, enumerate,color}
\usepackage{bbm}
\newtheorem{thm}{Theorem}

\newtheorem{fact}{Fact}
\newtheorem{lemma}[thm]{Lemma}

\newtheorem{conj}[thm]{Conjecture}

\def\beq{ \begin{equation} }
\def\eeq{ \end{equation} }
\def\mn{\medskip\noindent}

\def\ep{\epsilon}

\def\square{\vcenter{\vbox{\hrule height .4pt
  \hbox{\vrule width .4pt height 5pt \kern 5pt
        \vrule width .4pt} \hrule height .4pt}}}

\def\CC{\mathbb{C}}
\def\RR{\mathbb{R}}
\def\ZZ{\mathbb{Z}}

\def\var{\hbox{var}\,}

\newcommand{\ind}[1]{\mathbf1\{#1\}}
\renewcommand{\vec}{\mathbf }

\usepackage{graphicx}
\graphicspath{%
    {converted_graphics/}% inserted by PCTeX
    {/}% inserted by PCTeX
}

\title{Poisson percolation on the oriented square lattice}
\author{Irina Cristali, Matthew Junge, and Rick Durrett}

\begin{document}

\maketitle

\begin{abstract}
In Poisson percolation each edge becomes open after an independent exponentially distributed time with rate that decreases in the distance from the origin. As a sequel to our work on the square lattice, we describe the limiting shape of the component containing the origin in the oriented case. We show that the 
density of occupied sites at height $y$ in the cluster is close to the percolation probability in the corresponding homogeneous percolation process, and we study the fluctuations of the boundary.
\end{abstract}

\section{Introduction}

\emph{Percolation} was introduced by Broadbent and Hammersley a little over 60 years ago  to model a porous medium \cite{BH}. The model goes by including each edge of the integer lattice $\mathbb Z^d$ independently with probability $p$.  
One of the most fundamental questions is whether the subgraph contains an infinite component. 
There is known to be a critical value $p_c(d)$ such that for $p>p_c(d)$ such a component exists almost surely. A  vast amount of literature is devoted to understanding the geometry of this component for different values of $p$. See Grimmett's book \cite{grimmett} for a thorough introduction or
the article by Beffara and Sidoravicius \cite{BS} for a briefer overview. 

The subgraph obtained via homogeneous percolation is static. In \cite{poisson} we introduced Poisson percolation, which has a stochastically growing set of open edges. This could potentially model a medium that becomes more porous over time. 
Each edge in the unoriented square lattice $\mathbb Z^2$ with midpoint $x\in \mathbb Z^2$ becomes open at rate $\|x\|_{\infty}^{-\beta}$.
%  (Note that what we call `$\beta$' here was denoted by `$\alpha$' in \cite{poisson}. Keeping with convention, we will use `$\alpha$'  to denote the limiting speed of the rightmost edge in oriented percolation.) 
Thus, the probability an edge is open at time $t$ is equal to
$
\rho(x,t) = 1 -  \exp( -t \|x \|^{-\beta}).
$
We studied three aspects of the structure of $\mathbb C_0$, the connected component containing $0$. 
See Figure \ref{fig:sim} for a simulation.

\mn
{\bf Size and shape of the cluster.}
For fixed $t$, the probability an edge beyond distance $N = t^{1/\beta} (\log2)^{-1/\beta}$ is open is smaller than $p_c(2) = 1/2$. Accordingly, we show \cite[Theorem 1]{poisson} that with high probability $\mathbb C_0 \subseteq (1+\epsilon)[-N,N]^2$ for all $\epsilon >0$. 

\mn
{\bf Cluster density.} Fix $1/2<a<1$ and tile $(1-\epsilon)[-N, N]^2$ with boxes $R_{i,j}$ with side-length $N^a$. In \cite[Theorem 2]{poisson} we  proved that with high probability the density in each box $|\mathbb C_0 \cap R_{i,j}|/N^{2a}$ is close to the density of the giant component in homogeneous percolation with parameter $\rho(x_{i,j},t)$.

\mn
{\bf Fluctuations of the boundary.} The fluctuations of $\mathbb C_0$ in the $e_1$ direction are defined as  $
|N-\max\{x \colon (x,0) \in \mathbb C_0(t)\}|.
$ 
Our understanding of this quantity comes from work of Nolin \cite{nolin} on gradient percolation, in which the probability $p$ that a bond (or site) is open decreases linearly from 1 to 0 as the height is increased from 0 to $N$. Rather than edge percolation on a square lattice, he considered site percolation on the triangular lattice in order to take advantage of the rigorous computation of critical exponents by Smirnov and Werner \cite{SW}. In this setting, he studied the shape of the boundary of the cluster containing the base of a trapezoidal region of length $\ell_N$ and height $N$. He found that the edge stays in a strip of width $N^{4/7+\delta}$ centered at $N/2$, and the length of the front is $\ell_N^{(3/7)\pm \delta}$. These results are expected to hold on the square lattice. In our system the change in the density is nonlinear but smooth. Since the position of the front is dictated by bonds that are open with probabilities close to $1/2$, the boundary behavior should be the same.

\begin{figure} % float placement: (h)ere, page (t)op, page (b)ottom, other (p)age
  \centering
  % file name: C:/Users/Rick/Dropbox/PoissonPercolation/OPpaper/CSq.jpg
  \includegraphics%[width = 8 cm]
  [width=5.67in,height=5.52in,keepaspectratio]
  {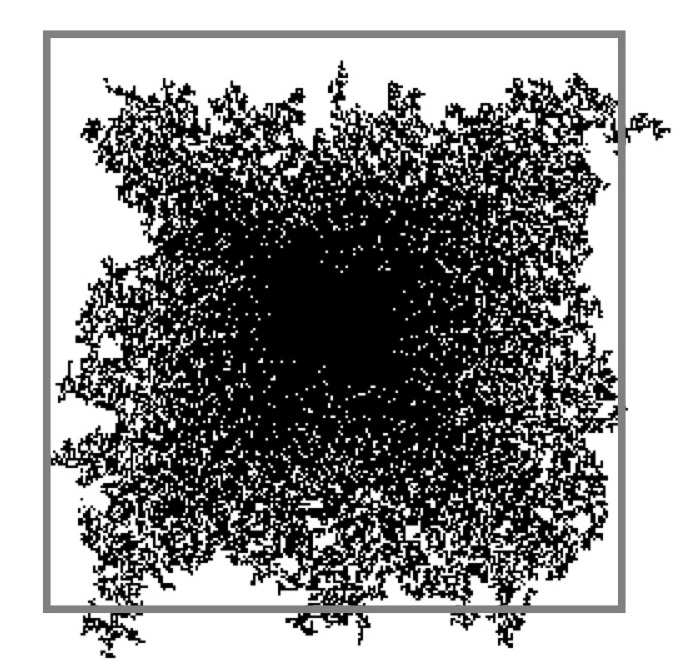}
  \caption{$\mathbb C_0$ on the unoriented square lattice for $\beta=1$ and $t = 104$. The gray box has radius $N =150$.}
  \label{fig:sim}
\end{figure}

%\begin{figure} % float placement: (h)ere, page (t)op, page (b)ottom, other (p)age
%  \centering
%  % file name: C:/Users/Rick/Dropbox/PoissonPercolation/OPpaper/CSq.jpg
%  \includegraphics%[width = 8 cm]
%  [width=5.67in,height=5.52in,keepaspectratio]
%  {MJ_edit.png}
%  \caption{$\mathbb C_0$ on the oriented square lattice for $\beta=1/2$ and $t = 1000$. The horizontal line is at  height $N\approx 839$.}
%  \label{fig:sim}
%\end{figure}

In this article we study Poisson percolation on oriented lattice ${\cal L} = \{ (m,n) \in \ZZ^2 \colon m+n \hbox{ is even} \}$
with oriented edges from $(m,n) \to (m+1,n+1)$ and $(m,n) \to (m-1,n+1)$. This is $\mathbb Z^2$ rotated $45^\circ$.
We will again study the size and shape of the cluster, its density, and edge fluctuations. The oriented case has
fewer symmetries so the shape is more interesting than a square (see Figure \ref{fig:op}).

\begin{figure}
\begin{center}
\includegraphics[width = 12 cm]{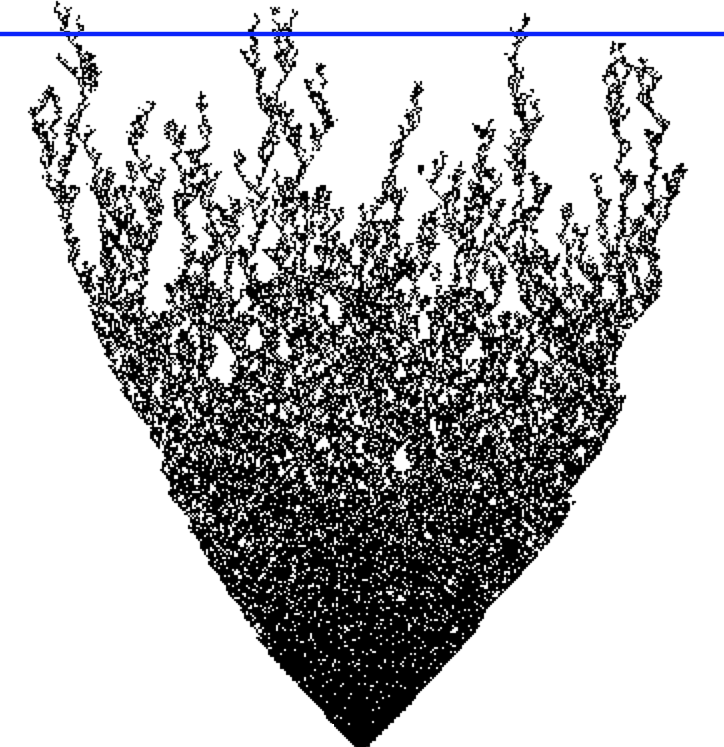}	
\end{center}
	\caption{$\mathbb C_0$ on the oriented lattice with $\beta =1/2$ and $t=30$. The blue line is at height $N\approx 839$. } \label{fig:op}
\end{figure}

Fix $\beta>0$. An edge with midpoint $(x,y)$ and $y>0$  is open with probability 
\beq
\rho(y,t) = 1 - \exp(-ty^{-\beta}).
\label{rhoyt}
\eeq
Let $n(p,t)= \max\{ y \colon \rho(y,t) \ge p \}$ be the largest height at which edges are open with probability $\ge p$. 
A little algebra shows that 
\beq
n(p,t) \sim c_{p,\beta}t^{1/\beta}\quad\hbox{where}\quad c_{p,\beta} = (-\log(1-p))^{-1/\beta}.
\label{npt}
\eeq
We write $(x,m) \to (y,n)$ if there is a path of open edges from $(x,m)$ to $(y,n)$. Let $$\CC_0(t) = \{ (x,n) \colon (0,0) \to (x,n) \},$$
and let
\beq
f(y)  = \rho(yt^{1/\beta},t) = 1 - \exp(-y^{-\beta}).
\label{foy}
\eeq
Define $y_c$ by $f(y_c) = p_c$, where $p_c \approx .64470019$ (see page 5242 of \cite{Jensen}) is the critical value $\inf\{ p\colon  P( |{\cal C}_0| = \infty \}$
where ${\cal C}_0$ is the cluster containing the origin. 

\subsection{Size and shape of the cluster}

To define the limiting shape of $\mathbb C_0$ we need to introduce the \emph{right-edge speed} in homogeneous percolation. Consider homogeneous percolation with parameter $p$ on oriented ${\cal L}$. Following \cite{RD80}, we let 
$$
r_k = \sup \{ x \colon \exists y\leq 0  \text{ with } (y,0) \to (x,k) \}
$$ 
be the right most site at height $k$ that can be reached from a site in $(-\infty,0] \times \{0\}$. The subadditive ergodic theorem guarantees the existence of a limiting speed $ r_k / k \to \alpha(p)$ for $p \ge p_c$. Obviously $\alpha(1)=1$. It is known that $\alpha(p_c)=0$. When $p<p_c$ the system dies out
exponentially fast so $\alpha(p)=-\infty$.

Letting $g(0)=0$ and $g'(y) = \alpha(f(y))$ for $0 \le y \le y_c$ we define our limiting shape
$$
\Gamma = \{ (x,y)\colon |x| < g(y), 0 \le y \le y_c \} \subseteq \mathbb R^2.
$$
Intuitively the shape result is
$$
\mathbb C_0(t) / t^{1/\beta} \to \Gamma.
$$

To prove the result it is convenient to work on the unscaled lattice. 
Let $r_t(k) = \max\{x \colon (x,k) \in \mathbb C_0(t)\}$ be the right edge at height $k$ at time $t$, 
and let $\ell_t(k) = \min\{x \colon (x,k) \in \mathbb C_0(t)\}$. 
It is convenient to have $g$ defined for all $y>0$ so we let $g(y) = g(y_c)$ for $y> y_c$.
Let 
$$
\Gamma_t(y) = t^{1/\beta}g(y/t^{1/\beta}).
$$
Throughout the paper we will let $N=n(p_c,t)$.

\begin{thm} \label{thm:ooutside}
For any $\eta > 0$, as $t \to\infty$, 
\begin{enumerate}[(i)]
	\item $P( \CC_0(t) \subset \RR \times [0,(1+\eta) N]) \to 1$, and
\item 
$
P\Bigl(  -(1+\eta)\Gamma_t(k) \le \ell_t(k), \  r_t(k) \le  (1+\eta) \Gamma_t (k) \text{ for all $k \le (1+\eta)N$ }\Bigr) \to 1 .$
\end{enumerate}
\end{thm}

Proving Theorem 1 (i) is easy because our percolation process is subcritical when $y>N$. To prove Theorem 1 (ii) we fix $m$ (it does not grow with $t$) and decompose $\mathbb Z \times [0,(1+\eta)N]$ into $m$ strips, $\mathbb Z \times [z_i,z_{i+1})$, so that $\alpha_i = \alpha( \rho(z_i,t) )= 1- i/m$ for $i <m$. We dominate the process in each strip by using homogeneous percolation with probability $p_i = \rho(z_i,t)$. Large deviation estimates on the distance of the right edge from $\alpha$ from \cite{RD84} allow us to prove that $\mathbb C_0$ lies to the left of a piecewise linear function whose slope is $\alpha_i$ in each strip. 

The next result gives a corresponding lower bound.

\begin{thm} \label{thm:RE}
For any $\eta >0$, as $t\to\infty$
$$
P( \ell_t(k) \le -(1-\eta)\Gamma_t(k) \text{ and }   (1-\eta) \Gamma_t(k) \le r_t(k) \text{ for all $k \le (1-\eta)N$ }) \to 1.
$$
\end{thm}

Again we divide space into strips $\mathbb Z \times [z_i,z_{i+1})$, but now we lower bound the process by using homogeneous percolation with probability $p_i = \rho(z_i,t)$. In each strip we use a block construction to relate our process to a 1-dependent oriented bond percolation on $\ZZ^2$ with parameter $p =1-\epsilon$. On the renormalized lattice the right edge has speed close to 1. When we map the path of the right edge back to the Poisson perecolation process we get a piecewise linear function that serves as a lower bound on the location of $r_t(k)$. 

\subsection{Cluster density}

Let $P_p$ be the probability measure for oriented bond percolation on $\ZZ^2$, when edges are open with probability $p$. Let $\mathcal C_0$ be the open cluster containing the origin, and let $\theta(p) = P_p( |{\cal C}_0| = \infty)$. Let
$$
G(t,\eta) = \cup_{y=0}^{(1-\eta)N} [-(1-\eta)\Gamma_t(y), (1-\eta)\Gamma_t(y)] \times \{ y \}.
$$ 
Intuitively, our next result says that near $(x,y) \in G(t,\eta)$ the density of points in $\CC_0(t)$ will be close to $\theta(\rho(y,t))$. To state this precisely, fix $1/2 < a < 1$ and tile the plane with boxes of side length $N^{a}$: 
$$
R_{i,j} = [iN^a,(i+1)N^a] \times [jN^a,(j+1)N^a],
$$
and let $(x_{i,j},y_{i,j})$ be the center of $R_{i,j}$. Let $D_{i,j} = |\CC_0(t) \cap R_{i,j}|/N^{2a}$ be the fraction of points in $R_{i,j}$ that belong to ${\CC}_0(t)$ and let $\Lambda(t,\eta) = \{ (i,j) \colon R_{i,j} \subset G(t,\eta) \}$ be the indices of boxes that fit inside $G(t,\eta)$.

\begin{thm} \label{thm:oinside}
For any $\eta,\delta>0$, as $t\to\infty$,
$$
P\left( \sup_{(i,j) \in \Lambda(t,\eta)} |D_{i,j}(t) - \theta(\rho(y_{i,j},t))| > \delta \right) \to 0.
$$
\end{thm}

\subsection{Boundary fluctuations}

The first three results were laws of large numbers. We will now consider the fluctuations of the right edge $r_t(k)$. In the homogeneous case, Galves and Presutti \cite{GalPre} were the first to prove such a central limit theorem for the supercritical contact process. Their proof also applies to oriented percolation. It implies that, if we start with the nonpositive integers occupied, then there is a constant $\sigma(p)$ so that for all $k > 0$ as $n \to \infty$
$$
\frac{r_{kn} - \alpha(p)kn }{\sigma(p)\sqrt{n}} \Rightarrow B_s.
$$
Here $B_s$ standard Brownian motion and $\Rightarrow$ is weak convergence of stochastic processes.
Two years later Kuczek \cite{CLT} simplified the proof by introducing what he called break points: times $T_i$ at which 
the right-most particle starts an oriented percolation that does not die out. In this case for $i \ge 1$,
the increments $(r_{T_{i+1}}-r_{T_i}, T_{i+1}-T_i)$ are i.i.d. Using his method we prove the analogue for Poisson percolation.

\begin{thm} \label{edgef}
As $t \to \infty$,
$$
\frac{r_t(\lceil Nu\rceil) - \int_0^{Nu} \alpha(p(y,t))dy}{ N^{1/2} } \,   \Rightarrow W_u,
$$ 
where $W_u, 0 \le u < 1$ is a Gaussian process with independent increments. It holds that $EW_u=0$ and
$$
EW_u^2 = \frac{1}{N} \int_0^{Nu} \sigma^2(p(y,t)) \, dy.
$$
\end{thm}

\noindent
Given the result for the homogeneous case this conclusion is what one would expect to hold; if we divide the space into a large number of thin strips
we have a sequence of homogeneous oriented percolation processes

Very little is known rigorously about critical exponents for oriented percolation, so we are not able to prove mathemtically an analogue of
Nolin's result. However, we can give a physicist style derivation of the following:

\begin{conj}
Fluctuations in the height of $\CC_0(t)$ are of order $N^{0.634}$.
\end{conj}

 First, recall that oriented percolation has two correlation lengths.
The correlation length in time, $L_{\parallel}$ can be defined simultaneously for the subcritical and supercritical cases by
$$
\gamma_{\parallel}(p) = \lim_{t\to\infty} \frac{1}{n} \log P( n \le \tau^0 < \infty ) \qquad L_{\parallel}(p) = 1/\gamma_{\parallel}(p).
$$
The correlation length in space $L_{\perp}$ has two different definitions for $p<p_c$ and $p> p_c$. Let $R^0 = \sup\{ x : x \in \xi^0_T \hbox{ for some $x$}\}$ and define
\begin{align*}
\gamma_{\perp}(p) = \lim_{t\to\infty} \frac{1}{n} \log P( R^0  \ge n) & \qquad L_{\perp}(p) = 1/\gamma_{\perp}(p) \\
\gamma_{\perp}(p)  = \lim_{t\to\infty} \frac{1}{n} \log P( \tau^{\{-n,\ldots n\}} < \infty ) & \qquad L_{\perp}(p) = 1/\gamma_{\perp}(p).
 \end{align*}
The last two limits and the one that defines $\gamma_{\parallel}$ when $p<p_c$ exist due to supermultiplicativity (e.g.\ $P( R^0 \ge m+n ) \ge
P( R^0 \ge m ) P( R^0 > n)$). See \cite{DST} for more details, and some other definitions. 

The corresponding critical exponents are defined by
$$
L_{\parallel}(p) \approx |p-p_c|^{-\gamma_{\parallel}}  \qquad L_{\perp}(p) \approx |p-p_c|^{-\gamma_{\perp}} .
$$
Here $\approx$ could be something as precise as $\sim C |p-p_c|^-{\gamma}$ or $\log L(p)/\log|p-p_c| \to \gamma$.
Numerically, see \cite[equation (15)]{Jensen}
$$
\gamma_{\parallel} = 1.733847 \qquad 2 \gamma_{\perp} = 2.193708.
$$

Nolin gives the following ``hand-waving'' argument for his result \cite[page 1756]{nolin}. If we are at distance $N^b$ behind the
front then $p - p_c = O(N^{b-1})$ and the correlation length is 
$$
|p-p_c|^{-\nu_{parallel}} = O(N^{(1-b)\nu_{parallel}})
$$
if $b=(1-b)\nu_{parallel}$, i.e., $b = \nu_{parallel}/(1+\nu_{parallel})$, then the correlation length matches the distance behind the 
front. In this case the physical interpretation of the correlation length implies that the percolation process will look like the
critical case. Nolin's proof of the localization of the front, see \cite[Theorem 6]{nolin}, is not long, but it is based on properties of
sponge crossing, which will be difficult to generalize to the oriented case. However, there has been some recent work in that
direction by Duminil-Copin et.\ al.\ \cite{DCetal}.

\section{Percolation toolbox}

Here we state additional definitions and facts that we will need in the proofs of our theorems. The first is a simple observation that percolation is monotonic in the parameter. 

\begin{fact} \label{fact:mono1}
Let $G_p\subseteq \mathbb Z^2$ be the random subgraph obtained in homogeneous oriented percolation with parameter $p$. There exists a coupling such that if $p<p'$, then $G_p \subseteq G_{p'}.$	
\end{fact}

This follows by coupling the Bernoulli random variables on each each edge in the canonical way. A similar statement holds in Poisson percolation.

\begin{fact} \label{fact:mono2}
Let $G(t)$ be the set of all open edges at time $t$ in Poisson percolation. Fix a subset of edges $H \subseteq \mathbb Z^2$ and let $$p^- = \min \{ \rho(\vec x,t) \colon \vec x \in H)\}, \quad p^+ = \max \{ \rho(\vec x,t) \colon \vec x \in H\}.$$
There exists a coupling such that $G_{p^-} \cap H \subseteq G(t) \cap H \subseteq G_{p^+} \cap H.$ 
\end{fact}  

The estimate in \cite[(1) Section 7]{RD84} bounds the probability that there is a path from $\ZZ \times \{0\}$ to $\ZZ \times \{k\}$.

\begin{fact} \label{fact:subpath}
 Let $\xi_k^0$ be the set of vertices in $\ZZ \times \{k\}$ that connect to $\ZZ \times \{0\}$. For any $\delta >0$. there is a constant $\gamma = \gamma(\delta)>0$ so that 
$$
P_{p_c-\delta}(\xi_k^0 \neq \emptyset  ) \le e^{-\gamma k}.
$$
\end{fact}

We are also interested in the speed of the rightmost particle in supercritical homogeneous percolation where we assume all edges in $(-\infty,0] \times \{0\}$ are open. \cite[(2) Section 7]{RD84} gives the following estimate.

\begin{fact} \label{fact:edgeLD}
If $p> p_c$ and $\beta>\alpha(p)$, then there are constants $0< \gamma, C < \infty$ which depend on $p$, and $\beta$ so that
$$
P_p( r_k > \beta n ) \le C e^{-\gamma k}.
$$
\end{fact}

We will make use of the \emph{dual process} to oriented homogeneous percolation when proving Theorem 2. This is the process obtained by keeping the same edges open, but reversing the orientation of $\mathbb Z^2$ so that edges point southwest and southeast. This will be useful because we can deduce an edge is likely to be open by following its dual process for $O(\log n)$ steps. Let $\tau^w$ denote the vertical distance covered by the cluster started at $w$ in the dual process. Results in \cite[Section 12]{RD84} imply that, for homogeneous percolation with parameter $p$,  the dual cluster at a site is exponentially unlikely to be large and finite.

\begin{fact} \label{fact:survival}
$P(k \leq \tau^{w} < \infty) \leq C e^{-\gamma k}.$
\end{fact}

Note that the dual process has the same law as usual oriented percolation. Thus, Fact \ref{fact:survival} also holds for the vertical distance of a component started at $x$ in the usual homogeneous oriented percolation. 

Supercritical percolation almost surely contains an infinite component. Translation invariance of the lattice ensures that an individual edge has probability $\theta(p)$ of belonging to this component. Despite correlations between the inclusion of edges in this component, subsets $H \subseteq \ZZ^2$ are exponentially likely as a function of their size to intersect the infinite cluster. This is proven in \cite[Section 10]{RD84}. Let $\tau^H$ denote the length of the longest surviving path started from an edge in $H$.

\begin{fact} \label{fact:subsets}
There exists $0<\gamma,C<\infty$ that depend on $p$ such that for any $H \subseteq \ZZ^2$ it holds that
$$P(\tau^H < \infty ) \leq C e^{-\gamma |H|}.$$
\end{fact}

Our proof involves \emph{one-dependent oriented percolation}. One-dependence means that the values on edges that share a common vertex are correlated, but edges without a common vertex are independent. This type of percolation is analyzed in \cite{RD84}. Consider one-dependent oriented percolation in which the marginal distribution for each edge is such that it is open with probability at least $1-\ep$.
Let ${\cal C} = \{ w \colon \hbox{ for some } x \le 0, (x,0) \to w \}$, and let $s_k = \sup\{ x \colon (x,k) \in {\cal C} \}$. According to \cite[Theorem 2; Section 11]{RD84}, 

\begin{fact} \label{fact:reld}
If $0< q< 1$ and $\ep < 3^{-36/(1-q)}$, then there are constants $0 < \gamma, C < \infty$ so that
$$
P( s_k \le qk ) \le Ce^{-\lambda k}.
$$
\end{fact}

%We consider 1-dependence because it comes up when we do a block construction. The ``blocks" are parallelograms of height $L$ and constant slope that have a small amount of overlap. We delay the 

\section{Proof of the Theorem \ref{thm:ooutside}}

We start by proving (i). Let $\delta >0$ be small. For $i=1,2$, let $y_i = \lceil n(p_c-i\delta,t) \rceil$. On $(y_1,\infty)$ we use Fact \ref{fact:mono1} to dominate Poisson percolation by homogeneous percolation in which bonds are open with probability $p_c-\delta$. We have $y_i \sim c_i t^{1/\alpha}$ for constants $c_1 \leq c_2$. Let $k = y_2-y_1$. Note that at height $y_1$, all the $x$-coordinates of points of $\CC_0(t)$ are in $[-y_1,y_1]$. It follows from
 Fact \ref{fact:subpath} that for large $t$
\beq
P( \CC_0(t) \cap (\ZZ \times \{y_2\}) \neq \emptyset ) \le 2c_1 t^{1/\alpha} \exp( - \gamma (c_2-c_1) t^{1/\alpha} ) \to 0.
\label{bhgt}
\eeq
If $\delta$ is small, then  $y_2 < (1+\eta)N$ and we have the desired upper bound on the height.

Theorem \ref{thm:ooutside} (ii) follows from the following two lemmas.
We subdivide time by introducing probabilities $p_i$, $1 \le i \le m-1$ so that $\alpha(p_i) = 1 - i/m$, and let $p_0=1$,  $p_m=p_c-2\delta$. We will choose the value of $m$ appropriately for $\eta$ in just moment.  
Let $z_i = n(p_i,t)$. The last interval $(z_{m-1},z_m]$ is longer so that $z_m = y_2$.
When $z_{i} < y \le z_{i+1}$, we use Fact \ref{fact:mono2} to bound our system from above by oriented percolation with probability $p_i$, which has edge speed
$=  1 - i/m$. 

We define sequences $u_i$, $v_i$ for $0 \le i \le m-1$ inductively by $u_0 = \delta$
$$
v_i = u_i + (1-i/m)(z_{i+1}-z_i), \qquad u_{i+1} = v_i + \delta.
$$
Now define a function $h_t(x)$ to be linear on $[z_i,z_{i+1})$, with $h_t(z_i) = u_i$ and 
$$
\lim_{y \uparrow z_{i+1}} h_t(y):= h_t(z_{i+1}-)   = v_i.
$$

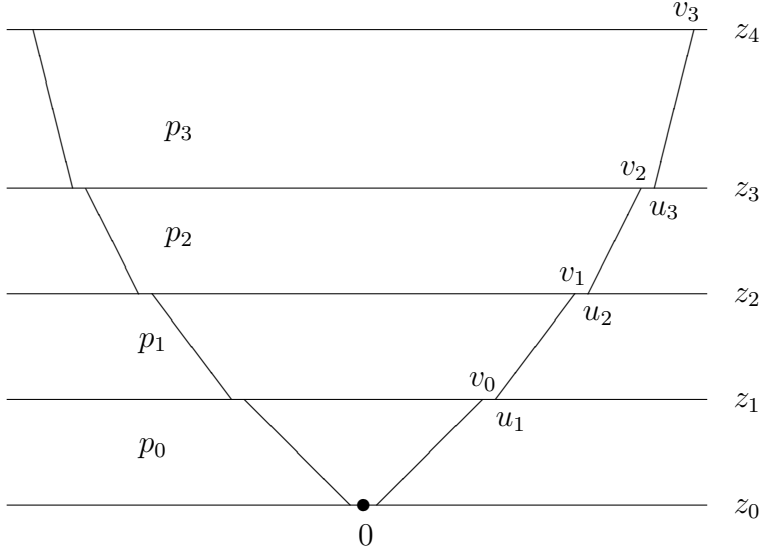
\begin{figure}[h]
\begin{center}
\begin{picture}(320,220)
\put(147,27){$\bullet$}
\put(148,15){0}
\put(155,30){\line(1,1){40}}
%\put(155,20){$u_0$}
\put(190,75){$v_0$}
\put(200,70){\line(3,4){30}}
\put(200,60){$u_1$}
\put(224,115){$v_1$}
\put(235,110){\line(1,2){20}}
\put(233,100){$u_2$}
\put(247,155){$v_2$}
\put(260,150){\line(1,4){15}}
\put(258,140){$u_3$}
\put(267,215){$v_3$}
\put(145,30){\line(-1,1){40}}
\put(100,70){\line(-3,4){30}}
\put(65,110){\line(-1,2){20}}
\put(40,150){\line(-1,4){15}}
\put(15,210){\line(1,0){265}}
\put(290,207){$z_4$}
\put(15,150){\line(1,0){265}}
\put(290,147){$z_3$}
\put(15,110){\line(1,0){265}}
\put(290,107){$z_2$}
\put(15,70){\line(1,0){265}}
\put(290,67){$z_1$}
\put(15,30){\line(1,0){265}}
\put(290,27){$z_0$}
\put(65,50){$p_0$}
\put(65,90){$p_1$}
\put(75,130){$p_2$}
\put(75,170){$p_3$}
\end{picture}
\caption{Region defined by $h_t(k)$ when $m=4$. Notice that the slopes 
of the segments $(u_i,v_i)$ are 1, 4/3, 2, and 4, i.e., 1 over the maximum edge speed in the interval.}
\end{center}
\label{fig:ht}
\end{figure}

\begin{lemma} \label{xub}
As $t\to\infty$,
$
P( r_k(t) \le h_t(k) \hbox{ for all $k\le z_m$}) \to 1.
$
\end{lemma}   

\begin{proof} 

Let $1 \le i < m$. Suppose that $r_{z_i}(t) \le v_{i-1}$. To prove the result it is enough to show that as $t\to\infty$
\begin{align}
P( r_k(t)  \le h_t(k) \hbox{ for $z_i \le k < z_{i+1}$ }) \to 1.  \label{eq:hind}
\end{align}
When $i=0$, the dominating process has $p_0=1$ so 
$$
P( r^0_k \le h_t(k) \hbox{ for $z_0 \le k < z_{1}$ }) = 1.
$$

Now suppose $i>0$. When $k \in [z_i,z_i + u_i-v_{i-1})$, it is impossible for the process to reach $h_t(k)$ since the $x$-coordinate of the right-most particle can increase by at most 1 on each step. 
In order to get from $v_{i-1}$ to $v_i$ in time $z_{i+1}-z_i$, the right edge would have to travel at an average speed of more than $ 1 - (i-1)/m$. Using
Fact \ref{fact:edgeLD}, and summing over $k \in [z_i + u_i-v_{i-1},z_{i+1}]$ proves \eqref{eq:hind}.
\end{proof}

\begin{lemma} \label{hvsgamma}
Let $\eta>0$. If we take $m$ large enough and $\delta$ small then $h_t(y) \leq (1+ \eta) \Gamma_t(y)$ for all 
$y \in [0,(1+\eta)N]$.
\end{lemma}

\begin{proof} 
We begin by noting that Fact \ref{fact:mono2} implies that  $\alpha(f(z))$ is decreasing while $f(z)>p_c$.
If $m$ is large enough then $\alpha(p_i)-\alpha(p_{i-1})< \eta/2$ for $1\le i < m$
and $\alpha(p_{m-1}) = 1/m < \eta/2$. To prove the result now note that
if $i<m$ then
\begin{align*}
\Gamma_t(z_i) - \Gamma_t(z_{i-1}) &= \int_{z_{i-1}}^{z_i} \alpha(f(y)) \, dy \\
h(z_i -) - h(z_{i-1}) &= \alpha(f(z_{i-1}))(z_i - z_{i-1})  .
\end{align*}
So, by the choices we have made above,
$$
h(z_i-) - h(z_{i-1}) < (1+\eta/2)(\Gamma_t(z_i) - \Gamma_t(z_{i-1})).
$$
Now, if $\delta$ is small enough $h(y) < (1+\eta/2)\Gamma_t(y)$ for $y<z_{m-1}$.
On $[z_{m-1},z_m]$, 
$$
h(y)-h(z_{m-1}) < (\eta/2)(y-z_{m-1}).
$$ 
If $\delta$ is small enough we have $h(y) < (1+\eta)\Gamma_t(y)$ for $y<(1+\eta)N$. 
\end{proof}

\section{Proof of Theorem \ref{thm:RE}}

To get the cluster at $0$ started, we observe that it with high probability contains all possible sites within distance $t^{b/\beta}$ with $0<b <1.$ 
%This also implies that the rightmost site expands at rate $1$ up to height $t^{b/\alpha}$. 

\begin{lemma} \label{solid}
Let ${\cal K}(n) = \{ (x,y) \colon 0 \le y \le n, \hbox{and }|x| \le y \}$. 
For any $0<b <1$, as $t \to \infty$ 
$$P(\mathcal K (t^{b/\beta}) \subseteq \mathbb C_0) \to 1.$$

\end{lemma} 

\begin{proof} By \eqref{rhoyt}, each edge in  ${\cal K}(t^{b/\beta})$ is closed with probability $\le \exp(-t^{1-b})$. Since there are fewer than
$t^{2b/\beta}$ edges, the result follows from a union bound.
\end{proof}

\subsection{Constructing the renormalized lattice} \label{sec:renorm}

The next step is renormalizing the lattice to compare Poisson percolation with 1-dependent oriented percolation with parameter $1-\epsilon$. 
As in the previous section, we introduce probabilities $p_i$, $1 \le i \le m-1$ so that $\alpha(p_i) = 1 - i/m$. We let $z_0 = t^{b/\beta}$
and for $1 \le i \le m-1$ let $z_i = n(p_i,t)$. The key ingredient for describing the density is to show that the rightmost edge of $\mathbb C_0$ stays to the right of $(1- \eta) \Gamma$. When $1 \le i \le m-1$ and $z_{i-1} < y \le z_{i}$, we bound our system from below by oriented percolation in which edges are open with 
probability $p_i$, and the edge speed is $\alpha_{i} = 1 -i/m$. 

To lower bound the process in which each edge is open with probability $p_i$ we will use a block construction. 
So that the lattices associated with different strips will fit together nicely, the $x$ coordinates of the sites in
the renormalized lattice will always be at integer multiples of some fixed constant $L$, and we will vary the heights of the blocks. In the $i$th strip $z_{i-1} < y \le z_{i}$, we
let $A^i_{0,0}$ be the parallelogram with vertices
\begin{align*}
u_0 = (-1.5\delta L,0) & \qquad u_1 = ((1+1.5\delta)L, (1+3\delta)L/\alpha_i) \\
v_0 = (-0.5\delta L,0) & \qquad v_1 = ((1+2.5\delta)L, (1+3\delta L)/\alpha_i)
\end{align*}
and let $B^i_{0,0} = -A^i_{0,0}$,

To begin to define the renormalized lattice, we let $T_1=z_0$. In the $i$th strip, the points in the renormalized lattice are
$$
(c^i_m,d^i_n) = (mL,T_i + n(1+\delta)L/\alpha_i)
$$
where $m$ and $n$ are integers so that $m+n$ is even, $n \ge 0$ and $T_i+n(1+3\delta)L/\alpha_i \le z_i$. 
The last condition is to
guarantee that all the edges we consider in the $i$th part of the construction are open with probability at least $p_i$.
Note that in each strip the vertical index $n$ begins at 0. 
 
To continue the construction when $i<m-1$ we let
$$
T_{i+1} = \max\{ T_i + n(1+\delta)L/\alpha_i :  T_i+n(1+3\delta)L/\alpha_i \le z_i \}.
$$
Let $A^i_{m,n} = (c^i_m,d^i_n) + A^i_{0,0}$, let $B^i_{m,n} = (c^i_m,d^i_n) + B^i_{0,0}$ and let $I^i_{m} = c^i_m +(-0.5\delta L,0.5\delta L)$.
The parallelograms are designed so that (see Figure \ref{fig:block})

\begin{enumerate}[(i)]
	\item at height $d^i_{n+1} = d^i_n + (1+\delta L)/\alpha$, $A^i_{m,n}$ fits in $I^i_{m+1}$.

		\item at height $d^i_n + (1+3\delta L)/\alpha$ the $x$ component of the left edge of $A^i_{m,n}$ is 
the same as that of the right edge of $B^i_{m+1,n+1}$.

\end{enumerate}

\begin{figure}[h]
\begin{center}
\begin{picture}(160,250)
\put(30,30){\line(1,2){80}}
\put(40,30){\line(1,2){80}}
\put(30,30){\line(1,0){10}}
\put(70,77){$A^i_{m,n}$}
\put(80,27){$c^i_n$}
\put(110,190){\line(1,0){10}}
\put(135,187){$c^i_n + (1+3\delta L)/\alpha_i$}
\put(50,30){\line(-1,2){40}}
\put(60,30){\line(-1,2){40}}
\put(50,30){\line(1,0){10}}
\put(42,28){$\star$}
\put(92,148){$\star$}
\put(100,150){\line(1,0){10}}
\put(120,147){$c^i_n + (1+\delta)L/\alpha_i$}
\put(100,150){\line(-1,2){40}}
\put(110,150){\line(-1,2){40}}
\put(80,150){\line(1,0){10}}
\put(80,150){\line(1,2){40}}
\put(90,150){\line(1,2){40}}
\put(5,225){$B^i_{m+1,n+1}$}
\put(135,225){$A^i_{m+1,n+1}$}
\end{picture}
\caption{Picture of the block construction. Stars mark points of the renormalized lattice. }
\label{fig:block}
\end{center}
\end{figure}
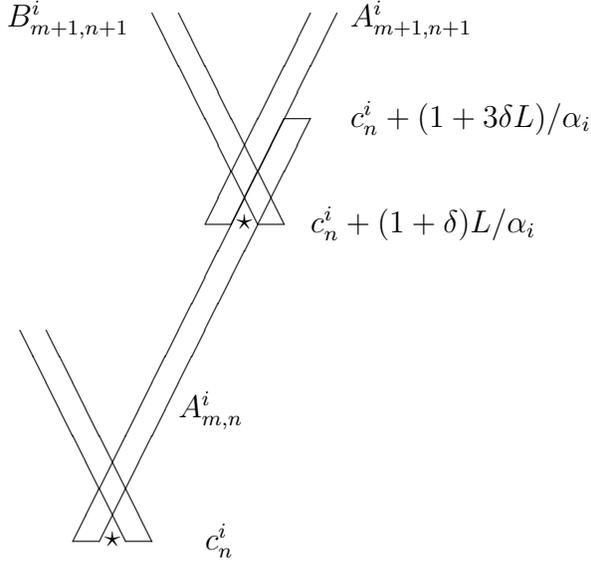

\medskip
We say that the good event $G^i_{0,0}$ occurs if 
\begin{enumerate}[(I)]
\item in $A^i_{0,0}$ there is a path from the bottom edge to the top edge.

\item in $B^i_{0,0}$ there is a path from the bottom edge to the top edge.
\end{enumerate}

Note that the existence of the paths in (I) and (II) are determined by the edges in $A^i_{0,0}$ and $B^i_{0,0}$ respectively.
The parallelograms are constructed to overlap in such a way (see Figure \ref{fig:block}) that, if there is a path in $A^i_{m,n}$, and there are paths in $B^i_{m+1,n+1}$ and $A^i_{m+1,n+1}$, then there is a path from the bottom edge of $A^i_{m,n}$ to the top edges of $A^i_{m+1,n+1}$ and $B_{m+1,n+1}$.

We define $G^i_{m,n}$ by translation. In  \cite[Section 9]{RD84} it was shown that, given $\ep>0$ ,
for $L \ge L_i$ it holds that $P(G^i_{0,0}) \ge 1-\ep$. Let $\bar L = \max_{1\le i \le m-1} L_i$.
Suppose $\delta< 0.01$, let $R^i_{0,0} = [-1.5L,1.5L] \times [0,(1+3\delta)L/\alpha_i]$, and let
$$
R^i_{m,n} = (c^i_m,c^i_n) + R^i_{0,0}.
$$ 
The existence of paths in parallelograms that do not overlap is independent.
The box $R^i_{0,0}$ intersects $R^i_{2,1}$, $R^i_{-2,1}$, $R^i_{-1,0}$,
$R^i_{1,0}$, $R^i_{2,-1}$, and $R^i_{-2,-1}$, so the construction is one dependent (as described above Fact \ref{fact:subsets}). 
\subsection{Lower bound for the right-most particle} \label{sec:RE}

%As noted in the Section \ref{sec:renorm}, the vertical index $n$ starts from 0 in each strip. 
To facilitate comparison with oriented percolation, we will renumber the rows of renormalized sites with $z_0 \le y \le z_{m-1}$ by the nonnegative integers $0, 1, 2, \ldots M   $
and let $\tau_0, \tau_1, \ldots ,\tau_M=\inf \{ k \colon z_k \geq (1-\eta)N\}$ be the corresponding heights in Poisson percolation on the usual lattice. In our construction, we will pick 
$L$ large and then let $t\to\infty$, so there are constants $C_1$ and $C_2$ so that
$C_1 t^{1/\beta} \le M \le C_2t^{1/\beta}$. Also, fix $0<b<1$ and let $K=K(t,b) = \max \{ j \colon \tau_j \leq t^{b/\beta}\}$ be the last parallelogram below height $t^{b/\beta}.$ Note also that $K \to \infty$ since $L$ is fixed.

Consider 1-dependent oriented percolation in which edges are open with probability $1-\ep$. Fix two numbers $0<q<q'<1$. Define the set of edges $\mathcal E_K = [q'K,K] \times \{0\}$, and 
$$s'_k = \max \{ x \colon \text{ there exists $w \in \mathcal E_ K$ with } w \to (x,k)\}$$
to be the rightmost edge at height $k$ accessible from $\mathcal E_K$. By Lemma \ref{solid} we know that $\mathcal E_K$ will have all edges open with probability going to 1. Moreover,  we claim that as $t \to \infty$,
\begin{align}P(s_k' \geq qk \text{ for all $k \geq 0$}) \to 1.\label{eq:right}
\end{align}
Fact \ref{fact:subsets} guarantees that the probability $\mathcal E_K$ contains a path to infinity goes to 1 as $t \to \infty$. Since a path can displace at most one unit to the left at each height, the first time we could have $s_k' < qk$ is at height $(q'-q)K/2 $.  Applying the bound from Fact \ref{fact:reld} to the rightmost edge started from $\mathcal E_K$, we then have
 $$P(s_k' \leq qk \text{ for some $k\geq 0$} ) \to \sum_{k= (q'-q)K/2}^M C e^{-\gamma k} \to 0,$$
 since $K \to \infty$.

To get a lower bound on the right-edge in the Poisson percolation process, we consider the mapping
$
(s_k',k) \to (Ls_k',\tau_k)
$
 from the renormalized lattice back to the original lattice. Because of \eqref{eq:right}, we consider the image of the line $y=qk$ under this map. It is given by a piecewise linear function with
\begin{itemize}

  \item 
$h(0) = 0$, and $h(t) = qk$ for $k \in [0,z_0]$, and
\item
 $h(k) = h(z_{i-1}) + q\alpha_i(k-z_{i-1})$  for  $k \in [z_{i-1}, z_i]$ with $1\leq i \le m-1$.
\end{itemize}

The renormalized sites that make up the right edge will map to the right of this curve. The paths that connect them will
lie in the associated parallelogram from Section \ref{sec:renorm}, so they cannot go further than $(1+3\delta)L/\alpha_i$ to the left of $h$. It follows that 
$$
P( r_t(k) \ge h(k) - (1+3\delta)L/\alpha_i \hbox{ for all } z_{i-1} \le t \le z_{i} ) \to 1.
$$
%Choosing $m$ large enough as in Lemma \ref{hvsgamma}, this proves the result for the right edge. 
On $[z_{i-1},z_i]$, $h$ has slope $q\alpha_i$ while $\Gamma_t$ increases at rate $\le \alpha_{i-1} = \alpha_i + 1/m$.
If $m$ is large enough then $\alpha_i \ge (1-\eta/2)\alpha_{i-1}$ for $1 \le i \le M$. 
%[Here we need the truncation or else
%we would have a situation where $\alpha_i = 1/m$ and $\alpha_{i-1}=2/m$.] 
It follows that if $q$ is chosen close enough to 1
then $h(k) - (1+3\delta)L/\alpha_i \ge (1-\eta)\Gamma_t(k)$ for all $z_{i-1} \le k \le z_i$ and $1\le i \le m-1$. The proof for the left edge is similar. 
%
%\blue{[Do we need a different lemma to choose $m$ so that the piecewise curve is to the right of $\Gamma$? - MJ]}

\section{Proof of Theorem \ref{thm:oinside}}

Consider the site $w = (x, y)$ with $z_i \leq y < z_{i+1}$, so that it is in the $i${th} strip of the unscaled lattice.
Fact \ref{fact:survival} implies 

\mn
($\star$) if $n_i = (1/ \gamma_i) \log(C_i N^4)$ and the dual process started from $w$ survives for $n_i$ units of time
then the probability $w \not\in \mathbb C_0$ is $ \le 1/N^4$. 

\mn
%Since the number of points in $G(t,\eta)$ is $O(N^2)$, $
This says that $\mathbb C_0$ is closely approximated by the points whose dual survives for time $n_i$. Let
$$
R_{j,k} = [jN^a,(j+1)N^a] \times [kN^a,(k+1)N^a]
$$ 
and suppose that all the points in $R_{j,k}$ are in the $i$th strip.

Let $A_{w} = \{\tau^w \geq  k_i\},$
and count the number of points in $R_{j,k}$ with a long-surviving dual with
\begin{align*}
S_{j,k} = \sum_{w \in R_{j, k}} \ind{A_w}.
\end{align*}
Since $|R_{j,k}| = N^{2s}$, ($\star$) ensures that
\begin{align}
P(S_{j,k} \neq |\mathbb C_0 \cap R_{j,k} |) \leq N^{2a-4}.\nonumber 
\end{align}
Since there are no more than $N^{2-2a}$ boxes with high probability this holds for all of them.

So, it suffices to study $S_{j,k}$. We start by centering it. Let $\theta_w = P( A_w)$, and define 
\beq
\bar{S}_{j,k} = S_{j,k} - ES_{j,k} = \sum_{w \in R_{j,k}} \ind {A_w} - \theta_w. \nonumber 
\label{center}
\eeq
The advantage of considering $A_w$ is that if $w=(x,y)$ the event $A_w$ is determined by edges in $[x-n_i,x+n_i]\times[y-n_i,y]$ so if $\|w-w'\| > 2n_i$
the indicator random variables are independent. Using the bound 
$$
|\ind{A_w\cap A_{w'}} - P(A_w)P(A_{w'})| \leq 1
$$  
when  $\|w - w'| \leq 2n_i$, we obtain
\beq
{E} \bar S_{j,k}^2  \leq N^{2a}\cdot 4n_i^2 \leq \frac{4}{\gamma_i} N^{2a} \log^2{C_i N^4}.
\label{secmoment}
\eeq

Using \eqref{secmoment} with Chebyshev's inequality gives for $\delta >0$ and some $C_i'>0$
\beq
P(|\bar{S}_{j,k}| >  \delta N^{2a}) \leq \frac{C_i' N^{2a} \log N}{\delta^2 N^{4a}} = O( N^{-2a} \log N ).
\label{eq:cheby}
\eeq
Since there are $O(N^{2 - 2a})$ many different boxes $R_{j,k}$, it follows from \eqref{eq:cheby} that 
\begin{align}
P\left( \sup_{(j,k) \in \Lambda(\eta,\eta)} |\bar S_{j,k}| > \delta N^{2a} \right) = O(N^{2 - 4a} \log N) . \nonumber %\label{eq:last_step}
\end{align}
The right term is $o(1)$ since $a > 1/2$.
To relate this back to $\mathbb C_0$ we note that $f(y)$ defined in \eqref{foy} is Lipschitz continuous
and $\theta(p)$ is on $[p_c+\delta,1]$ so  
$$
\sup \{ |\theta(\rho(w, t)) - \theta(\rho(w', t))|\colon w,w' \in R_{j,k} \} \le C N^{a-1} \to 0.
$$
Using this and Fact \ref{fact:survival}, we can replace the $P(A_w)$ terms in $S_{j,k}$ with a representative 
$\theta_{j,k} = \theta(\rho((x_j,y_k),t))$, and Theorem \ref{thm:oinside} follows.

\section{Proof of Theorem \ref{edgef}}

Recall $N = n(p_c,t)$. 
In our process, the right edge particle cannot be part of an infinite cluster, so we define renewals to be times at which the rightmost particle lives for time at least $\log^2 N$. This is motivated by the bound from Fact \ref{fact:survival}. To get started, if $b<1$ then the state at time $t^{b/\alpha}$ is an interval and the rightmost particle survives for $\log^2N$ with probability $\to 1$ by Lemma \ref{solid}. Suppose $t_i$ is the time of the $i$th renewal and let $p_i$ be the probability bonds are open at that time. On $[t_i,t_i+2\log^2N]$ 
bonds are open with probability $\ge p_i - c (\log^2 N)/N$. The 2 is to allow us to find the renewal point and then verify it works.
The bonds of interest are in a triangle with point at $(r_i,T_i)$, sides with slope 1, and height $2 \log^2 N$ 
so we can by Fact \ref{fact:mono2} couple the inhomogenenous system with a system with probabilities $p_i$ so that with high probability there are no errors. 

Unfortunately the increments in the right-edge defined in this way are not independent. If $r_i-r_{i-1}$ is large then the $p$ for the next increment
will be smaller. To fix this we will again divide $[0,N]$ into strips by choosing $\alpha(p_i) = 1 -i/m$ and $z_i = n(p_i,t)$ but now we will
use $m = N^{0.6}$ strips. For renewals that begin in the strip $z_i \le y < z_{i+1}$ we will upper bound the movement of the right edge by using $p=p_i$
and lower bound by using $p=p_{i+1}$. The large number of strips guarantees that the difference between the upper and lower bounds on 
$E(r_k-r_{k-1})$ will be $N^{-0.6}$ so when we sum $N$ of these terms the result is $O(N^{0.4}) = o(N^{0.5})$

Kuczek \cite{CLT} has shown that when $p$ is fixed $r_i-r_{i-1}$ has an exponential tail,
so using the Lindberg-Feller theorem, see e.g., Theorem 3.4.5 in \cite{PTE4}, on the upper bound and on the lower bound 
\beq
\frac{\sum_{k=1}^n (r_k-r_{k-1}) - E(r_k-r_{k-1})}{\sqrt{ \sum_{k=1}^n \var(r_k - r_{k-1}) }} \Rightarrow \chi
\label{discCLT}
\eeq
where $\chi$ is standard normal.
To convert this to continuous time note that for homogeneous percolation
\begin{align*}
E(r_i-r_{i-1})  = \alpha(p) E_{p}(t_i-t_{i-1}) & \quad\hbox{because $Er(t)/t \to \alpha(p)$}, \\
\var(r_i-r_{i-1})  = \sigma^2(p)E_{p}(t_i-t_{i-1}) & \quad\hbox{because $\var r(t)/t \to \sigma^2(p)$}.
\end{align*}
Let $M(s)$ be the number of renewals needed to get to height $s$. Replacing $n$ by $M(s)$ in \eqref{discCLT} the result is
$$
\frac{r(s) - \int_0^s \alpha(p(y,t)) \, dy }{\int_0^s \sigma^2(p(y,t)) \, dy } \Rightarrow \chi.
$$

Taking $s=Nu$ and replacing the denominator by $\sqrt{N}$ we have convergence of the one dimensional distributions to the desired limit.
Since the increments of the limit process are independent, convergence of finite dimensional distributions follows easily. Since
$$ 
\sum_{k=1}^n (r_k-r_{k-1}) - E(r_k-r_{k-1})
$$
is a square integrable martingale it is not hard to use the $L^2$ maximal inequality to check that the tightness criteria that can be found for 
example in Section 8 of Billingsley \cite{Bill}. Alternatively one can invoke Theorem 4.13 on page 322 of Jacod and Shiryaev. \cite{JS}.

%\clearp

\end{document}